\documentclass[12pt]{extarticle}
\usepackage{amsmath, amsthm, amssymb, color}
\usepackage[colorlinks=true,linkcolor=blue,urlcolor=blue]{hyperref}
\usepackage{graphicx}
\usepackage{caption}
\usepackage{mathtools}
\usepackage{enumerate}
\usepackage{verbatim}
\usepackage{tikz,tikz-cd,tikz-3dplot}
\usepackage{amssymb}
\usetikzlibrary{matrix}
\usetikzlibrary{arrows}
\usepackage{algorithm}
\usepackage[noend]{algpseudocode}
\usepackage{caption}
\usepackage[normalem]{ulem}
\usepackage{subcaption}
\usepackage{url}
\oddsidemargin \evensidemargin
\usepackage{makecell}
\usepackage{array}

\newtheorem{theorem}{Theorem}
\newtheorem{proposition}[theorem]{Proposition}
\newtheorem{lemma}[theorem]{Lemma}
\newtheorem{corollary}[theorem]{Corollary}

\theoremstyle{definition}
\newtheorem{definition}[theorem]{Definition}

\newtheorem{remark}[theorem]{Remark}
\newtheorem{conjecture}[theorem]{Conjecture}

\newtheorem{example}[theorem]{Example}

\newtheorem*{definition*}{Definition}
\numberwithin{theorem}{section}

\newcommand{\PP}{\mathbb{P}}
\newcommand{\RR}{\mathbb{R}}

\newcommand{\CC}{\mathbb{C}}

\title{\bf The Two Lives of the Grassmannian}

\author{Karel Devriendt, Hannah Friedman, \\ Bernhard Reinke and Bernd Sturmfels}

\date{}
\begin{document}
\maketitle

\begin{abstract}
  \noindent
  The real Grassmannian is both a projective variety  (via Pl\"ucker coordinates) and
    an affine variety (via  orthogonal projections).
    We connect these two representations, and we develop the
    commutative algebra of the latter variety.
    We introduce the
    squared Grassmannian, and we
    study applications to determinantal point processes in statistics.
\end{abstract}

\section{Introduction}

Linear subspaces of dimension $d$ in the vector space $\RR^n$
 correspond to points in the Grassmannian
 ${\rm Gr}(d,n)$, which is a real manifold of dimension $d(n-d)$.
 This manifold is an algebraic variety. This means that it can be
 described by polynomial equations in finitely many variables.
Such a description is useful for working with linear spaces in
applications, such as optimization on manifolds in data science  \cite{GTB} and
scattering amplitudes in  physics~\cite{lauren22}.

When reading about Grassmannians in 
applied contexts, such as \cite{GTB, kassel2019, LLY, LWY},
one learns that there are two fundamentally
different representations of ${\rm Gr}(d,n)$ as an algebraic 
variety. These are the ``two lives''  in our title.
 First, there is the Pl\"ucker embedding, which
realizes ${\rm Gr}(d,n)$ as a projective variety. 
The ambient space is the projective space
$\RR \PP^{\binom{n}{d}-1}$. In this embedding, ${\rm Gr}(d,n)$ is
described by a highly structured system of
quadratic forms, known 
as the Pl\"ucker equations \cite[Chapter 5]{MS}.
Second, there is the representation of  a linear space
by the orthogonal projection onto it. This realizes ${\rm Gr}(d,n)$
as an affine variety. Its ambient space is the
 space $\RR^{\binom{n+1}{2}}$ of
symmetric $n \times n$ matrices $P$.
In that embedding, ${\rm Gr}(d,n)$ is
described by the quadratic equations $P^2 = P$
and the linear equation ${\rm trace}(P) = d$.

The purpose of this article is to connect these
two lives of the Grassmannian ${\rm Gr}(d,n)$
and to explore some consequences of this connection
for algebraic geometry and its applications.
We begin with a simple example that 
explains the two different versions of the Grassmannian.

\begin{example}[$n=5,d=2$] \label{ex:52}
We consider the row span of a $2 \times 5$ matrix that has rank $2$:
$$ A \,\, = \,\, \begin{bmatrix} 
a_{11} & a_{12} & a_{13} & a_{14} & a_{15} \\
a_{21} & a_{22} & a_{23} & a_{24} & a_{25}  \end{bmatrix}. $$
This row space is a point in the $6$-dimensional manifold
${\rm Gr}(2,5)$. The embedding into $\RR \PP^9$ is given by 
 the ten $2 \times 2$ minors
$\, x_{ij} \, =\, a_{1i} a_{2j} - a_{1j} a_{2i}$.
These satisfy the Pl\"ucker equations 
\begin{equation}
\label{eq:plucker25} \!\! \begin{small}
 \begin{matrix} & x_{12} x_{34} - x_{13} x_{24} + x_{14} x_{23}  & = &
x_{12} x_{35} - x_{13} x_{25} + x_{15} x_{23}  & = & 
x_{12} x_{45} - x_{14} x_{25} + x_{15} x_{24} \\    = & 
x_{13} x_{45} - x_{14} x_{35} + x_{15} x_{34}   &  = & 
x_{23} x_{45} - x_{24} x_{35} + x_{25} x_{34} & = & 0 . \end{matrix}
\end{small}
\end{equation}
We note that
these five quadrics are the $4 \times 4$ Pfaffians in the skew-symmetric $5 \times 5 $ matrix
\begin{equation}
\label{eq:skew5}
X \,\, = \,\, \begin{small}
\begin{bmatrix}
\,\,\,0 & \phantom{-} x_{12} & \phantom{-} x_{13} & \phantom{-} x_{14} & \phantom{-} x_{15}  \\
- x_{12} & \,\,\,0 &  \phantom{-} x_{23} & \phantom{-} x_{24} & \phantom{-} x_{25}  \\
- x_{13} & - x_{23}  &\,\,\,  0 & \phantom{-} x_{34} & \phantom{-} x_{35}  \\
- x_{14} & - x_{24}  &  - x_{34} &\,\,\, 0  & \phantom{-} x_{45}  \\
- x_{15} & - x_{25}  &  - x_{35} & -x_{45} &\,\,\,  0  
\end{bmatrix}.
\end{small}
\end{equation}
Hence, ${\rm Gr}(2,5)$ is the projective variety in $\PP^9$ whose
points are the above matrices $X$ of rank $2$, up to scaling.
The subspace represented by $X$ is the row space (or column space) of $X$.

The second life of ${\rm Gr}(2,5)$ takes place in the
affine space $\RR^{15}$ of symmetric $5 \times 5$ matrices
$$  P \,\, = \,\, 
\begin{bmatrix} 
p_{11} & p_{12} & p_{13} & p_{14} & p_{15} \\
p_{12} & p_{22} & p_{23} & p_{24} & p_{25} \\
p_{13} & p_{23} & p_{33} & p_{34} & p_{35} \\
p_{14} & p_{24} & p_{34} & p_{44} & p_{45} \\
p_{15} & p_{25} & p_{35} & p_{45} & p_{55} 
\end{bmatrix}.
$$
We want $P$ to represent the orthogonal
projection $ \RR^5 \rightarrow \RR^5$ onto the row space of~$A$,~i.e.,
\begin{equation}
\label{eq:projector}
P \,\, = \,\, A^T ( A A^T)^{-1} A .
\end{equation}
This is the parametric representation of ${\rm Gr}(2,5)$ as an affine variety in $\RR^{15}$. The equations~are
$$ P^2 = P \quad {\rm and} \quad {\rm trace}(P) = 2. $$
To be completely explicit, ${\rm Gr}(2,5)$ is the irreducible affine variety whose prime ideal equals
$$ \begin{matrix} \langle \,
p_{11}^2+p_{12}^2+p_{13}^2+p_{14}^2+p_{15}^2-p_{11}\,,\,\,
p_{11} p_{12}+p_{12} p_{22}+p_{13} p_{23}+p_{14} p_{24}+p_{15} p_{25}-p_{12}\,,\, \ldots \quad \phantom{x}
\\ \quad \ldots \, , \,\,
 p_{14} p_{15} +p_{24} p_{25}+p_{34} p_{35}+p_{44} p_{45}+p_{45} p_{55}-p_{45} \, , \,\,
 p_{15}^2+p_{25}^2+p_{35}^2+p_{45}^2+p_{55}^2-p_{55} \,,\\
p_{11} + p_{22} + p_{33} + p_{44} + p_{55} - 2 \,
\rangle. \end{matrix}
$$
This inhomogeneous ideal realizes ${\rm Gr}(2,5)$ as a $6$-dimensional variety of degree $40$ in $\RR^{15}$.
By contrast, the ideal from (\ref{eq:plucker25}) realizes ${\rm Gr}(2,5)$ as a
$6$-dimensional variety of 
degree $5$ in $\RR \PP^9$.
\end{example}

This article is organized as follows.
In Section 2, we present an explicit formula, valid for all $d$ and $n$, 
which writes the projection matrix $P$ in terms of the Pl\"ucker coordinates~$X$.
The squares of these $X$ coordinates are proportional to
the $d \times d$ principal minors of $P$
and parametrize the {\em squared Grassmannian} ${\rm sGr}(d,n)$ in $ \PP^{\binom{n}{d}-1}$.
We study the degree and defining equations of ${\rm sGr}(d,n)$ in Section 3.
In Section 4, we turn to probability theory and algebraic statistics,  by studying 
 determinantal point processes given by ${\rm sGr}(d,n)$. We offer a comparison
 with other statistical models represented by the Grassmannian
 and compute ML degrees using numerical methods; see \cite{BT, FSZ, ST}.
  Section 5 is devoted to the {\em projection Grassmannian},
  ${\rm pGr}(d,n)$, which is the embedding of ${\rm Gr}(d,n)$ into the affine space $\RR^{\binom{n+1}{2}}$.
  We prove that the ideal given by
  $P^2 = P$ is radical and is the intersection of $n+1$
  primes, each obtained by setting $\,{\rm trace}(P) = d$.
  Finally, in Section 6, we study the moment map from
the Grassmannian to $\RR^n$, with a focus on the fibers of this map, as shown in~Figures~\ref{fig:octafib}~and~\ref{fig:endler}.

The Grassmannian is ubiquitous within the mathematical sciences.
It has many representations and many occurrences, well beyond the two lives
seen in this paper. Notably, ${\rm Gr}(d,n)$ is a representable functor, a 
homogeneous space, a differentiable manifold,  and a metric space.
For the metric perspective see, e.g.  \cite{BZA, kozlov1997}.
In this paper, we focus on equational descriptions
of Grassmannians, with a view towards statistics and combinatorics.

This article is accompanied by software and data. These materials
are made available in the {\tt MathRepo} collection at MPI-MiS via \
\url{https://mathrepo.mis.mpg.de/TwoLives/}.

\section{From Pl\"ucker coordinates to projection matrices}

Every point in ${\rm Gr}(d,n)$ can be encoded by
a skew-symmetric tensor $X $ or by a projection matrix $P$.
The $\binom{n}{d}$ entries of $X$ are denoted
$x_{i_1 i_2 \ldots i_d}$ where $1 \leq i_1,i_2,\ldots, i_d \leq n$.
Using the sign flips coming from skew-symmetry, we obtain
linearly independent coordinates $x_I$ where $I = (i_1,i_2,\ldots,i_d)$
satisfies  $ i_1 < i_2  < \cdots < i_d$.
We regard $X$ as a point in projective space~$\PP^{\binom{n}{d}-1}$.
It  lies in ${\rm Gr}(d,n)$ if and only if it satisfies the quadratic Pl\"ucker relations.
In this case, $X$ is the vector of  $d \times d$ minors
of any $d \times n$ matrix $A$ whose rows span the~subspace.

The orthogonal projection $\RR^n \rightarrow \RR^n$ onto that subspace is given by
a symmetric  $n \times n$ matrix $P = (p_{ij})$.
 This matrix is unique, and it satisfies $P^2 = P$ and ${\rm rank}(P) = {\rm trace}(P) = d$.
We can compute $X$ from $P$ by selecting
$d $ linearly independent rows, placing them into a
$d \times n$ matrix $A$, and then taking $d \times d$ minors.
Our next result shows how to find $P$ from~$X$.

After posting our article, we learned that Theorem \ref{thm:PPlucker}
also appeared in recent work of Bloch and Karp
on the positive Grassmannian in integrable systems; see
\cite[Lemma 4.11]{BK}.

\begin{theorem} \label{thm:PPlucker}
The entries of the projection matrix $P$ are ratios of quadratic forms in Pl\"ucker coordinates.
Writing  $I$ and $K$ for subsets of
size $d$ and $d-1$ of
$\{1,\ldots,n\}$, we~have
$$ p_{ij} \,\, = \,\, \frac{ \sum_K x_{iK} \,x_{jK}}{ \sum_I x_I^2 }. $$
The sum in the numerator has $\binom{n}{d-1}$ terms,
and the sum in the denominator has $\binom{n}{d}$ terms.
\end{theorem}

\begin{proof}
Our point of departure is the formula for $P$ in (\ref{eq:projector}).
By the Cauchy-Binet formula,
$$ {\rm det}(A A^T) \,= \, \sum_I x_I^2 . $$
If we multiply the inverse of $A A^T$ by this sum of squares
then we obtain the adjoint of $A A^T$. This is a
matrix whose entries are the $(d-1) \times (d-1)$ minors
of $A A^T$. We can identify the adjoint with the
$(d-1)$st exterior power, using the
isomorphism $\RR^d \simeq \wedge_{d-1} \RR^d$
that sends  the basis vector $e_i$ to the basis vector 
$(-1)^i e_1 \wedge \cdots \wedge e_{i-1} \wedge e_{i+1} \wedge \cdots \wedge e_{d}$.
In symbols, we~have
$$ {\rm adj}(A A^T) \, = \, \wedge_{d-1} (A A^T) \,  = \,
\wedge_{d-1} A \cdot \wedge_{d-1} A^T  \, = \,
\wedge_{d-1} A \cdot (\wedge_{d-1} A)^T. $$
This implies that the scaled projection matrix equals
\begin{equation}
\label{eq:scaledproj}
(\sum_I x_I^2) \cdot P \, = \, A^T {\rm adj}(A A^T) A \,= \,
(A^T \wedge_{d-1} \! A ) \cdot (A^T \wedge_{d-1} \!A )^T. 
\end{equation}
It remains to analyze the  matrix 
$A^T \wedge_{d-1} \! A $. This matrix has format $n \times \binom{n}{d-1}$ and its rank is $d$.
We claim that its entry in row $i$ and column $K$ is equal to $x_{iK}$. Indeed,
$x_{iK}$ is the $d \times d$ minor of $A^T$ with row indices $iK$. The $(i,K)$ entry in the matrix product
$A^T \wedge_{d-1} \! A $ is computed by multiplying the $i$th row of $A^T$ with the $K$th column 
of $\wedge_{d-1} \! A$. This is precisely the Laplace expansion of the $iK$  minor of $A^T$
with respect to  $i$th row. This implies that the 
entry of (\ref{eq:scaledproj}) in row $i$ and column $j$ is equal to
$\sum_K x_{iK} \,x_{jK}$. This completes the proof.
 \end{proof}

\begin{remark}
The $n \times \binom{n}{d-1}$ matrix $ ( x_{iK} )$
is known as the {\em cocircuit matrix} of the linear space.
Its  entries are the Pl\"ucker coordinates and its image is precisely the
linear space.
If~$d=2$ and $n=5$ then the cocircuit matrix is the skew-symmetric $n \times n$ matrix $X$
we saw in (\ref{eq:skew5}). This generalizes to $d=2 $ and $n \geq 6$. For $d \geq 3$,
the  cocircuit matrix has more columns than rows.
Theorem \ref{thm:PPlucker} states in words that the
projection matrix of a linear space equals the
cocircuit matrix times its transpose, scaled by
the sum of squares of all Pl\"ucker coordinates.
\end{remark}

\begin{example}[$n=5,d=2$] 
In Example \ref{ex:52}, the $5 \times 5$ projection matrix of rank $2$  equals
\begin{equation}
\label{eq:P52} P \,\,=\,\, \frac{X X^T}{ \sum_{1 \leq i < j \leq 5} x_{ij}^2} \,\, = \,\, \frac{2\cdot XX^T}{{\rm trace}(X^TX)}. \end{equation}
For an algebraic geometer, this formula represents a birational isomorphism from
the Grassmannian to itself. The base locus, given by the
denominator, has no real points.
\end{example}

\begin{example}[$n=6,d=3$] 
The Grassmannian ${\rm Gr}(3,6)$ has dimension $9$. As a projective variety it
has degree $42$ in $\RR \PP^{19}$ and its ideal is generated by $35$ quadrics
in the $20$ Pl\"ucker coordinates $x_{ijk}$. The $15 $ cocircuits
of the subspace are the columns of the cocircuit matrix 
$$
X \,\, = \,\,  \begin{small} \begin{bmatrix}
          0  &             0 &   0           &   0          &   0           & x_{123} & x_{124}     & \,\,  \cdots  \,\,   & x_{146}    & x_{156} \\
          0  &  \!-x_{123} & \!-x_{124} &\! -x_{125} & -x_{126} &    0        &       0     & \cdots       & x_{246}   & x_{256} \\
  x_{123} &           0  &\! -x_{134} & \!-x_{135} & -x_{136} &    0         & \! -x_{234}  & \cdots & x_{346}   & x_{356}  \\
  x_{124} & x_{134}  &   0          & \!-x_{145} & -x_{146} & x_{234}  &      0            & \cdots   & 0  &           x_{456} \\
  x_{125} & x_{135}  & x_{145}  &   0          & -x_{156} & x_{235}  &  x_{245}        & \cdots   & -x_{456}    &  0 \\
  x_{126} & x_{136}  & x_{146}  & x_{156}  &   0          & x_{236}  &  x_{246}        & \cdots  & 0                &  0 \\
  \end{bmatrix}. \end{small}
$$
As an affine variety, the Grassmannian ${\rm Gr}(3,6)$ lives in $\RR^{21}$.
 This embedded manifold consists of all 
$6 \times 6$ symmetric matrices $P$ with $P^2 = P$ and ${\rm trace}(P) = 3$.
It has dimension $9$ and degree  $184$. Similar to (\ref{eq:P52}),
we can write the projection matrix $P$ in terms of the $20$ Pl\"ucker coordinates as
the matrix product
$X \cdot X^T$ divided by the sum of squares  $\sum_{1 \leq i < j < k \leq 6} x_{ijk}^2$.
\end{example}

This formula holds in general. Indeed, Theorem \ref{thm:PPlucker}
can always be rewritten as follows:

\begin{corollary}
The formula for
the projection matrix in terms of Pl\"ucker coordinates~is 
$$ P\,\,=\,\, d\cdot\frac{ XX^T}{{\rm trace}(X^TX)}. $$
Here $X = (x_{iK}) $ is the cocircuit matrix, which has format $n \times \binom{n}{d-1}$.
\end{corollary}

\noindent In conclusion, we have presented the formulas that
  connect the two lives of the Grassmannian.

\begin{example}\label{ex:cut space}
The cut space of an oriented graph $G$ with $n$ edges is the subspace 
of $\mathbb{R}^n$ spanned by vectors representing edge cuts. 
If $G$ is connected, then its dimension $d$ equals the number of vertices minus one.
Kirchhoff \cite{kirchhoff1847} gave a 
formula for the~projection $P$ onto the cut space in terms of  {\em spanning forests}, i.e., maximal acyclic subsets of edges.~We~have
\begin{equation}\label{eq:spanning forests}
p_{ij} \,\,=\,\, \frac{\sum_K (-1)^{\sigma_K(i,j)}}{\#\lbrace\text{spanning forests in $G$}\rbrace},
\end{equation}
where $K$ runs over  sets of edges such that $iK$ and $jK$ are spanning forests, and $\sigma_K(i,j)$ is $1$ if $ijK$ contains an oriented cycle which traverses both $i$ and $j$ along their orientation and is $0$ otherwise; 
see \cite{biggs1997}.
Using that any Pl\"{u}cker coordinate $x_I$ of the cut space is in $\lbrace 0,\pm 1\rbrace$ and is nonzero when $I$ is a spanning forest of $G$, equation (\ref{eq:spanning forests}) follows 
from Theorem~\ref{thm:PPlucker}.
\end{example}

\section{The squared Grassmannian}

The projection matrix $P$ representing a point in ${\rm Gr}(d,n)$
is symmetric of format $n \times n$ and has rank $d$.
It thus makes sense to examine the largest non-vanishing
principal minors of $P$.
For $I \in \binom{[n]}{d}$, let $P_I$ denote the square submatrix
of $P$ with row indices and column indices~$I$.

\begin{lemma} 
\label{lem:seven}
The $d {\times} d$ principal minors of $P$
are proportional to squared Pl\"ucker coordinates:
\begin{equation}
\label{eq:detPI}
{\rm det}(P_I) \,\, = \,\, \frac{x_I^2}{\sum_{J\in\binom{[n]}{d}} x_J^2}.
\end{equation}
\end{lemma}

\begin{proof}
Let $A$ be any $d \times n$ matrix that satisfies (\ref{eq:projector}).
We write $A_I$ for the $d \times d$ submatrix of $A$ with column indices $I$.
Then (\ref{eq:projector}) implies
$P_I = (A_I)^T \cdot (A A^T)^{-1} \cdot A_I$.
This is a product of three $d \times d$ matrices, so its
determinant is the product of the three determinants. We find
$$
 \det(P_I) \,=\, \frac{\det(A_I)^2}{\det(AA^T)}\,=\,\frac{\det(A_I)^2}{\sum_{J}\det(A_J)^2}\,=\,\frac{x_I^2}{\sum_J x_J^2}.
$$
Here, $ I$ and $J$ are index sets in $\binom{[n]}{d}$.
The middle equation uses the Cauchy-Binet formula.
\end{proof}

Lemma \ref{lem:seven} offers a link between the
two lives of the Grassmannian. In the next section,~we 
shall see its role in probability and statistics. 
In the present section, it leads us to introduce the squared Grassmannian
${\rm sGr}(d,n)$  and to study its homogeneous
ideal~$\mathcal{I}({\rm sGr}(d,n))$.
\begin{definition}[Squared Grassmannian]
  The {\em squared Grassmannian} ${\rm sGr}(d,n)$ is the image of the Grassmannian ${\rm Gr}(d, n) \subset \mathbb{P}^{\binom{n}{d} - 1}$ in its Pl\"{u}cker embedding under the map
$$
    {\rm Gr}(d,n) \to \mathbb{P}^{\binom{n}{d} - 1}, \,\,
    (x_I)_{I \in \binom{[n]}{d}} \mapsto (x^2_I)_{I \in \binom{[n]}{d}}.
$$
\end{definition}
The two most basic invariants of any projective variety are its dimension and its degree.

\begin{proposition} The squared Grassmannian $\,{\rm sGr}(d,n)$ satisfies
$$ {\rm dim}({\rm sGr}(d,n)) \,=\,  d(n-d) \quad {\rm and} \quad
{\rm degree}({\rm sGr}(d,n)) \,=\, 2^{(d-1)(n-d-1)} \cdot {\rm degree}({\rm Gr}(d,n)). $$
\end{proposition}

\begin{proof}
The self-map of  $\PP^m$ given by squaring all coordinates is $2^m$-to-one.
It preserves the dimension of any subvariety, because all of its fibers are finite.
This yields the first equation.

 For~the second equation, we factor the squaring morphism
${\rm Gr}(d,n) \rightarrow {\rm sGr}(d,n)$ through the
quadratic Veronese map, which is an isomorphism from
${\rm Gr}(d,n) $ to its Veronese square $\nu_2({\rm Gr}(d,n))$.
Note that the degree of $\nu_2({\rm Gr}(d,n))$ is
$\,2^{d(n-d)}  \cdot {\rm degree}({\rm Gr}(d,n))$.
 The projection from $\nu_2({\rm Gr}(d,n))$ 
 onto ${\rm sGr}(d,n)$ 
 deletes all mixed coordinates.
This map is $2^{n-1}$-to-one since 
each fiber is given by switching the signs independently in
the $n$ columns of the matrix $A$; cf.~Remark \ref{rmk:18}. So, the degree of ${\rm sGr}(d,n)$ is
the degree of $\nu_2({\rm Gr}(d,n))$ divided by $2^{n-1}$.
\end{proof}

The familiar formula 
(cf.~\cite[Theorem 5.13]{MS})
for the degree of the Grassmannian implies:

\begin{corollary} An explicit formula for the degree of the squared Grassmannian is
$$ {\rm degree}({\rm sGr}(d,n)) \,\, = \,\,
\frac{2^{(d-1)(n-d-1)} \cdot (d(n-d))!  }{ \prod_{j=1}^d j (j+1) (j+2) \cdots (j+n-d-1)}. 
$$
\end{corollary}

For the special case $d=2$, these degrees are scalings of the Catalan numbers:
\begin{equation}
\label{eq:catalan} \begin{matrix}
 {\rm degree}({\rm sGr}(2,n)) \,\, = \,\,
\frac{2^{n-3}}{n-1} \binom{2n-4}{n-2} \,\, = \,\, 2^{n-3} \cdot C_{n-2}.  \end{matrix}
\end{equation}

We now turn to the ideal of the squared Grassmannian, beginning with the case $d=2$.
The squared Pl\"ucker coordinates are denoted by $q_{ij} = x_{ij}^2$ for $1 \leq i < j \leq n$.
We write these coordinates as the entries of a symmetric $n \times n$ matrix
that has zeros on the main diagonal.

\begin{theorem} \label{thm:conca}
The prime ideal $\,\mathcal{I}({\rm sGr}(2,n))$ is generated by the $4 \times 4$ minors
of the matrix
\begin{equation}
\label{eq:sym44} 
Q \,\, = \,\,\begin{bmatrix}
0 & q_{12} & q_{13} & \cdots & q_{1n} \\
q_{12} & 0  & q_{23} & \cdots & q_{2n} \\
q_{13} & q_{23} & 0 & \cdots & q_{3n} \\
 \vdots & \vdots & \vdots &  \ddots &   \vdots \\
  q_{1n} & q_{2n} & q_{3n} & \cdots  & 0
  \end{bmatrix}.
  \end{equation}
  \end{theorem}

\begin{proof}
  Let $Q$ be a generic symmetric $n \times n$ matrix.
  Let $J$ be the ideal generated by the $ 4 \times 4$ minors of $Q$, and let $I$ be the ideal $\langle q_{11}, \ldots, q_{nn}\rangle $, in the polynomial ring $\CC[Q]$  with complex coefficients.  
  Then $V(I + J) = V(I) \cap V(J)$ is the variety of $n \times n$ matrices of rank at most $3$ with zeros on the diagonal. 
  We will show that $I + J$ is prime, that its projective variety contains ${\rm sGr}(2, n)$, and that the projective varieties have the same dimension. 

  We first make some observations about the ring $\CC[Q]/J$. 
  We will show that $\CC[Q]/J \cong \CC[Y^T Y]$ where $Y$ is a generic $3 \times n$ matrix. 
  This isomorphism is the map $[q_{ij}] \mapsto (Y^T Y)_{ij}$, where $[q_{ij}]$ is the image of $q_{ij}$ under the quotient.
  The fact that $\CC[Y^T Y]$ is the image of this map is then clear. 
  By the first isomorphism theorem for rings, it is then sufficient to show that $J$ is equal to the kernel of this map.
  The containment of $J$ into the kernel is relatively straightforward: the image of a $4 \times 4$ minor of $Q$ in $\CC[Y^T Y]$ is a $4 \times 4$ minor of $Y^TY$, which is zero as $Y^TY$ has rank at most $3$.
  The containment of the kernel into $J$ is the Second Fundamental Theorem of Invariant Theory for the orthogonal group; see \cite[Theorem 5.7]{DP}.  
  
  Now given some $3 \times n$ matrix $Y$, the product $Y^T Y$ is invariant under the action of the orthogonal group 
  $O(3) = O(3,\CC)$ on $Y$ by matrix multiplication on the left. 
  Thus we have $$\CC[Q]/J \,\cong \,\CC[Y^T Y] \,=\, \CC[Y]^{O(3)} \,\subset\, \CC[Y],$$ where $\CC[Y]^{O(3)}$ is the invariant ring of the action of $O(3)$ on $Y$ by left multiplication.
  We have justified the inclusion $\CC[Y^TY] \subseteq \CC[Y]^{O(3)}$.
  The other inclusion is given by the First Fundamental Theorem of Invariant Theory for the orthogonal group; see \cite[Theorem 5.6]{DP}. 
  As $O(3)$ is linearly reductive in characteristic zero,
   $\CC[Y]^{O(3)} = \CC[Y^TY]$ is a direct summand of $\CC[Y]$.
  In summary, $\CC[Q]/J$ is isomorphic to $\CC[Y^TY]$, which is a direct summand of $\CC[Y]$. 

  We now turn to the image of $(I + J)/J$ under this isomorphism.
  Recall that $\{[q_{11}], \ldots, [q_{nn}]\}$ generates $(I + J)/J$.
  The image of $(I + J)/J$ in $\CC[Y^TY] $ is generated by $\{(Y^TY)_{11},  \ldots, (Y^TY)_{nn}\}$.
  Call this ideal $I_{Y^TY} = \langle (Y^TY)_{11},  \ldots, (Y^TY)_{nn} \rangle \subseteq \CC[Y^TY]$.
  We now turn to the extension of this ideal to $\CC[Y]$, which we denote 
  $I_{Y^TY} \CC[Y] = \langle (Y^TY)_{11},  \ldots, (Y^TY)_{nn} \rangle \subseteq \CC[Y]$. 
  For all $i$, the quadric $(Y^TY)_{ii} = y_{1i}^2 + y_{2i}^2 + y_{3i}^2$ is irreducible.
  As the $(Y^TY)_{ii}$ use disjoint sets of variables, 
  we conclude that $I_{Y^TY} \CC[Y]$  is prime.

  We now show $I_{Y^TY}$ is prime. 
  Since $\CC[Y^TY]$ is a direct summand of $\CC[Y]$, we have $I_{Y^TY} = I_{Y^TY} \CC[Y] \cap \CC[Y^TY]$.
  The inclusion $\subseteq$
  holds, because $I_{Y^TY} \CC[Y] \cap \CC[Y^TY]$ is the contraction of the extension of $I_{Y^TY}$.
  For the inclusion $\supseteq$, suppose $a = \sum_{i=1}^m a_i b_i \in~ I_{Y^TY} \CC[Y] \cap \CC[Y^TY]$
  where $b_i \in \CC[Y]$ and $a_i \in I_{Y^TY}$.
  The $\CC[Y^TY]$-module homomorphisms $\CC[Y^TY] \xhookrightarrow{} \CC[Y] \xrightarrow{\pi} \CC[Y^TY]$ compose to the identity, so 
  $a  = \pi(a) = \sum_{i=1}^m a_i \pi(b_i) \in I_{Y^TY}$.
  This proves the other containment.
  To conclude that $I_{Y^TY}$ is prime, recall that $I_{Y^TY}$ is the inverse image of $I_{Y^TY} \CC[Y]$ under the map $\CC[Y^TY] \xhookrightarrow{} \CC[Y]$, and that
   inverse images of primes are prime.

  We now show that $V(I + J)$ has the correct dimension. 
  The dimension of the projective variety $V(J)$ is $3n-4$ because
  any symmetric matrix of rank $\leq 3$ is determined by the upper triangle of the
  upper left $3 \times 3$ block and the $3  (n-3)$ entries in the upper right corner.
  To find the dimension of $V(I + J)$, we note that
  $[q_{11}], \ldots, [q_{nn}]$ is a regular sequence in
  the ring $\CC[Q]/J$.
  One proves this by applying the arguments above to show that
  $\langle [q_{11}], \ldots, [q_{ii}] \rangle$ is prime for any $i$. 
  Therefore, the dimension of the projective variety
$V(I+ J)$ is $(3n - 4) - n = 2n - 4$. In particular,
$V(I+J)$ has the same dimension as the
 Grassmannian ${\rm Gr}(2,n)$.

 We now argue that the squared Grassmannian
   ${\rm sGr}(2,n) $ is contained in $V(I + J)$.
  For any point $q$ in ${\rm sGr}(2,n)$, 
  there is a
$2 \times n$ matrix $A = (a_{kl})$ such that
$q_{ij} = a_{1i}^2 a_{2j}^2 - 2 a_{1i} a_{2i} a_{1j} a_{2j} + a_{1j}^2 a_{2i}^2$ for $1 \leq i < j \leq n$.
Hence, $Q$ has zeros on the diagonal and is a sum of three matrices of rank one, meaning that
its $ 4 \times 4$ minors vanish. This shows that $q \in V(I + J)$.

Because $V(I + J)$ has the same dimension as the squared Grassmannian, ${\rm sGr}(2,n) \subseteq V(I + J)$, and both varieties are irreducible, it follows that ${\rm sGr}(2,n) = V(I + J)$.
Finally, because $I + J$ is prime, we may apply the Nullstellensatz to see that $\mathcal{I}({\rm sGr}(2,n)) = I + J$. 
\end{proof}

\begin{example}[$d=2, n \leq 5$]
The Grassmannian ${\rm Gr}(2,4)$ is defined in $\PP^5$ by the quadric
$ x_{12} x_{34} - x_{13} x_{24} + x_{14} x_{23} $.
By setting $q_{ij} = x_{ij}^2$ and eliminating the $x$-variables, we obtain
 \begin{equation}
 \label{eq:qquartic} q_{12}^2 q_{34}^2+q_{13}^2 q_{24}^2+q_{14}^2 q_{23}^2
 -2 q_{12} q_{13} q_{24} q_{34}-2 q_{12} q_{14} q_{23} q_{34}
 - 2 q_{13} q_{14} q_{23} q_{24}. \end{equation}
 This  quartic is the determinant of a
 symmetric $ 4 \times 4$  matrix $(q_{ij})$ with zeros
 on the diagonal. 
 
The squared Grassmannian ${\rm sGr}(2,5)$  has dimension
 $6$ in~$\PP^9$. Its degree is $20$, by (\ref{eq:catalan}).
The ideal of ${\rm sGr}(2,5)$ is generated by $15$ quartics, namely the 
$4 \times 4$ minors in Theorem~\ref{thm:conca}.
\end{example}

\begin{conjecture} \label{conj:sGrQuartics}
For all $n \geq d \geq 2$, the
prime ideal $\mathcal{I}({\rm sGr}(d,n))$ is generated by quartics.
\end{conjecture}

We verified this conjecture for $d=3,n=6$ using {\tt Macaulay2} \cite{M2}.
The variety ${\rm sGr}(3,6)$ has dimension $9$ and degree
$672$ in $\PP^{19}$. Its ideal is minimally generated by 
$285$ quartics.

After the submission of this article, we learned
that the squared Grassmannian had already been studied in a different
context by Al Ahmadieh and Vinzant in \cite[Section 6.2]{AV}.
Note that 
the quartic in (\ref{eq:qquartic}) appears explicitly in \cite[Example 6.5]{AV}.
Most importantly, it is proved in \cite[Corollary 6.4]{AV}
that ${\rm sGr}(d,n)$ is cut out by
quartic polynomials that are derived from the $2 \times 2 \times 2$ hyperdeterminant.
We can therefore state their result as follows.

\begin{theorem}[Al Ahmadieh -- Vinzant]
The set-theoretic version of Conjecture \ref{conj:sGrQuartics} is true.
\end{theorem}

\section{Statistical models}

In this section, we view the Grassmannian as a discrete statistical model
whose state space is the set $\binom{[n]}{d}$ of $d$-element subsets of
$[n] = \{1,2,\ldots,n\}$.
We are aware of three distinct formulations of such an 
algebraic statistics model
that have appeared in the literature.

\smallskip

First, there is the positive Grassmannian ${\rm Gr}(d,n)_{> 0}$, which is
defined by requiring that all Pl\"ucker coordinates $x_I$ are positive.
This semialgebraic set plays a prominent role at the interface
of combinatorics and physics \cite{lauren22}. 
This is naturally a statistical model, via the usual
identification of the
positive projective space $\RR \PP^{\binom{n}{d}-1}_{>0}$ 
with the probability simplex
$\Delta_{\binom{n}{d}-1}$.
In this model, the probability of  observing a
$d$-set $I = \{i_1,i_2,\ldots,i_d\}$ equals
$\,x_I\,/\sum_{J} x_J$.

\smallskip

The second model is the configuration space $X(d,n)$
which is obtained from ${\rm Gr}(d,n)_{> 0}$
by taking the quotient modulo the natural torus
action by the multiplicative group $\RR^n_{>0}$.
This model plays a prominent role in the study of
scattering amplitudes; see
\cite{ST}. In the special case $d=2$, this Grassmannian
model is the moduli space $\mathcal{M}_{0,n}$
of $n$ distinct labeled points on the line $\RR \PP^1$.
This is a linear model, whose likelihood geometry is well understood.

\smallskip

In this section, we focus on a third statistical model, 
namely the squared Grassmannian ${\rm sGr}(d,n)$
from Section 4. In this model,
the probability of  observing a
$d$-set $I = \{i_1,i_2,\ldots,i_d\}$ equals
$\,q_I = x_I^2\,/\sum_{J} x_J^2$.
We shall discuss the probabilistic meaning of this in some detail.

First, however, we show how
these three models differ, by comparing 
their maximum likelihood (ML) degrees for $d \leq 3$.
Recall that the ML degree of a model is the number of complex critical points 
of the log-likelihood function for generic~data; see, e.g.,~\cite{FSZ, HKS, ST}.

\begin{theorem} \label{thm:MLD}
The ML degrees of the three models on small Grassmannians are as follows:
\begin{equation}
\label{eq:MLD2}
\begin{matrix} 
d=2 & &  n=4 & n=5 & n = 6 & n=7 & n=8 & n = 9\\
\hbox{positive Grassmannian} && 4 &  22 & 156 &  1368 & 14400 & 177840 & \\
\hbox{squared Grassmannian} && 3 & 12 & 60 & 360 & 2520 & 20160 \\
\hbox{moduli space $\mathcal{M}_{0,n}$} && 1 & 2 & 6  & 24 & 120 & 720 \\
\end{matrix}
\end{equation}
\begin{equation}
\label{eq:MLD3}
\begin{matrix} 
d=3 & &  n=5 & n = 6 & n=7 & n=8 \\
\hbox{positive Grassmannian} &&   22 & 1937 &  \geq 499976 & ??  \\
\hbox{squared Grassmannian} &&  12 & 552 & 73440 &  ?? \\
\hbox{moduli space $X(3,n)$} &&  2 & 26  &  1272  & 188112  \\
\end{matrix}
\end{equation}
\end{theorem}

\begin{proof}[Proof and discussion]
For the third row in (\ref{eq:MLD2}), see \cite[Proposition 1]{ST}, which states that $\mathcal{M}_{0,n}$
has ML degree $(n-3)!$. For the first row in (\ref{eq:MLD2}), the first two entries appear in \cite[Problem 12]{HKS}.
All other entries for the positive Grassmannians are new.
They were
found by numerical computations
with the software {\tt HomotopyContinuation.jl} \cite{BT}.
The  third row in (\ref{eq:MLD3}) 
 was first derived
in the physics literature, and later proved rigorously in \cite{Ago}.
All numbers for the squared Grassmannians are new and 
also found numerically with {\tt HomotopyContinuation.jl}.
Our methods for this are discussed  in more detail below.

In a strict formal sense, 
the values we computed numerically for the ML degrees
are only lower bounds for these ML degrees.
See \cite[Section 4]{FSZ} for a discussion. In particular,
we believe that the
$n = 7$ entry in
(\ref{eq:MLD3}) for the positive Grassmannian 
is a strict lower bound.
\end{proof}

We record the following conjecture which arises 
from the second row in  (\ref{eq:MLD2}).

\begin{conjecture}
The ML degree of ${\rm sGr}(2,n)$ equals $(n-1)!/2$.
\end{conjecture}

What is remarkable about the tables  (\ref{eq:MLD2}) and (\ref{eq:MLD3})
is that the ML degree of ${\rm Gr}(d,n)$ exceeds
the ML degree of ${\rm sGr}(d,n)$, even though
 the former is defined by quadrics and the latter is defined by quartics.
 Also, the degree of ${\rm sGr}(d,n)$
 is $2^{(d-1)(n-d-1)}$ times the degree of
${\rm Gr}(d,n)$. This suggests that the
squared Grassmannian has a special structure in statistics.
We now argue that this is indeed the case:
it represents a determinantal point process (DPP).

Let $\mu$ be a probability measure on the set $2^{[n]}$ of all subsets of $[n]$. 
We write $\mathbf{X}\sim \mu$ for a random subset that is distributed according to $\mu$. 
Then $\mu$ is a {\em determinantal point process} if its correlation functions are given by the
principal minors of a real symmetric $n \times n$ matrix $K$. 
Namely, for a determinantal point process (DPP), we have the formulas
$$
\Pr[J\subseteq \mathbf{X}] \,\,\,= \,\sum_{J\subseteq I\subseteq [n]}\mu(I) \,\, = \,\, {\rm det}(K_J) \qquad {\rm for~all} \,\,\, J \in 2^{[n]}.
$$
The matrix $K$ is called the {\em kernel matrix} of the DPP. The probability measure $\mu$ can be obtained from its
correlation functions by the following alternating sum:
\begin{equation}
\label{eq:moebius}
\mu(I)\,\,\, = \sum_{I\subseteq J\subseteq [n]}(-1)^{\vert J\backslash I\vert} \cdot \Pr[J\subseteq {\bf X}].
\end{equation}
This is the M\"{o}bius inversion of $\Pr[J\subseteq {\bf X}]$ over the Boolean lattice of subsets of $[n]$, see \cite{kassel2019}. 
If the matrix ${\rm Id}_n - K$ is invertible then we set $\Theta = K({\rm Id}_n - K)^{-1}$ to rewrite (\ref{eq:moebius}) as follows:
$$ \mu(I) \,\, = \,\, \frac{{\rm det}(\Theta_I)}{{\rm det}({\rm Id}_n + \Theta)} \quad \hbox{for all} \,\, I \in 2^{[n]}.  $$
This is the parametrization of the DPP  used in the recent algebraic statistics study \cite{FSZ}.

Here we stick with the kernel matrix $K$. Note that a symmetric matrix $K$ is the
kernel matrix of a DPP  if and only if its eigenvalues lie
in the interval $[0,1]$; see \cite[Theorem 22]{hough2005}. 

In the present work we are  interested in the boundary case
 when all eigenvalues of $K$ lie in the two-element set $\{0,1\}$. This means that $K$ is an orthogonal projection matrix $P$, satisfying $P^2 = P$.
 These  models are known as {\em projection determinantal point processes}.
 Now, the measure $\mu$ is supported on $\binom{[n]}{d}$,
where $d = {\rm rank}(P) = {\rm trace}(P)$.
   By \cite[Lemma~17]{hough2005}, the measure of a projection DPP is given by the principal minors of size $d \times d$ in $P$:
$$ \mu(I) \,\,= \,\,\det(P_I) \,\,=\,\, \frac{x_I^2}{\sum_J x_J^2} \quad
\hbox{for all} \,\, I \in \binom{[n]}{d}. $$
This is precisely the formula 
in Lemma \ref{lem:seven}.
We thus summarize our discussion as follows:

\begin{corollary}
The projection DPP is the discrete statistical model on the state space $ \binom{[n]}{d}$ whose
underlying algebraic variety is the squared Grassmannian ${\rm sGr}(d,n)$.
\end{corollary}

The implicit representation of the projection DPP as a model in the probability
simplex $\Delta_{\binom{n}{d}-1}$ was our topic in Section 3.
Here we turn to likelihood geometry.
For any of our three statistical models on the Grassmannian, any data set
can be summarized in a vector $u = (u_I)_{I \in \binom{[n]}{d}}$
whose coordinates $u_I$ are nonnegative integers. The {\em log-likelihood function} is
$$ L_u \,\, = \,\, \sum_{I \in \binom{[n]}{d}} u_I \cdot {\rm log}( \mu(I)). $$
Maximum likelihood estimation (MLE) aims to maximize $L_u$
 over all model parameters.
In algebraic statistics, we do this by computing all complex critical points
of $L_u$, which amounts to solving a system of rational function equations in the
model parameters. The number of complex solutions is the {\em ML degree} of the model.
See \cite{Ago, FSZ, HKS, ST} and references therein.

Our point of departure is
 the parametrization $\CC^{d(n-d)} \rightarrow \PP^{\binom{n}{d}-1}$
of  the Grassmannian ${\rm Gr}(d,n)$, which is given 
by the maximal minors of the
$d \times n$ matrix $[ \, {\rm Id}_d \,\, Y \, ]$, where
$Y= (y_{i,j})$ is a $d \times (n-d)$ matrix of unknowns.
The ML degrees for ${\rm Gr}(d,n)$ in (\ref{eq:MLD2}) and (\ref{eq:MLD3}) 
were computed directly from this parametrization with {\tt HomotopyContinuation.jl},
as  in \cite{Ago, FSZ}.

\begin{remark} \label{rmk:18} The natural parametrization of ${\rm Gr}(d,n)$ is birational.
However, its extension to
the squared Grassmannian is $2^{n-1}$ to $1$. The fibers  of this parametrization of
${\rm sGr}(d,n)$
are obtained by flipping the signs of rows and columns in
the parameter matrix $Y = (y_{i,j}) $.
\end{remark}

Remark \ref{rmk:18} means that the natural parametrization is inefficient
when it comes to computing critical points. 
Computation times can be reduced by a factor of  up to $2^{n-1}$ if we use a birational
reparametrization of our model. We now present such a reparametrization.
The construction is similar to \cite[Section 4]{FSZ}. 
We shall identify a transcendence basis for the field of invariants
of the $(\mathbb{Z}/2 \mathbb{Z})^n$ action on
$Y $ that is given by sign changes in rows and columns.

\begin{lemma} \label{lem:field}
A transcendence basis for the field of invariants is given by the
squares of the $n-1$ matrix entries in the first row and column of $Y$,
and the quartics obtained by multiplying the entries in 
the $(n-d-1)(d-1)$ adjacent $2 \times 2$-submatrices. We denote this basis~by
\begin{equation}
  \label{eq:birationalpara} \begin{matrix} \alpha \,:=\, y_{1,1}^2\,,\,\,
    \beta_j \,:= \,y_{1,j}^2\,
\,\,\,\hbox{for}\, \,\,2 \leq j \leq n{-}d ,\,\,
\,\,  \gamma_i \,:=\, y_{i,1}^2
 \,\,\, \hbox{for}\,\, \,2 \leq i \leq d,
\smallskip \\
\kappa_{ij} \,:=\, y_{i,j} \,y_{i,j+1} \,y_{i+1,j} \,y_{i+1,j+1}   \quad
\hbox{for} \,\,
1 \leq i \leq d{-}1,\,
1 \leq j \leq n{-}d{-}1.
\end{matrix} 
 \end{equation}
\end{lemma}

\begin{proof}
We consider the lattice consisting of all integer
matrices of format $ d \times (n-d)$ with even row sums and
even column sums. The matrices in this lattice correspond to
Laurent monomials in our invariant field. Every matrix can be
interpreted as a walk in the complete bipartite graph $K_{d,n-d}$
whose edges are labeled by the unknowns $y_{i,j}$.
One proves by induction on $d$ and $n$ that every such walk
can be decomposed into the basis described above.
\end{proof}

Lemma \ref{lem:field} implies the following
formulas for inverting our birational parametrization.

\begin{proposition} \label{prop:inversion}
The unknowns $y_{i,j}$ are expressed as follows in terms of the invariants.
$$ y_{1,1} \,=\, \alpha^{1/2},\,\,
y_{1,j} \,=\, \beta_j^{1/2},\,\,
y_{i,1} \,=\, \gamma_i^{1/2}, \,\, {\rm and}
\vspace{-0.1in} $$
$$
y_{i,j} \,\,=\,\,
\alpha^{(-1)^{(i+j+1)}/2} \,\beta_j^{(-1)^{(i+1)}/2}\,\gamma_i^{(-1)^{(j+1)}/2}\,
\prod_{k=1}^{i-1} \prod_{l=1}^{j-1} \kappa_{kl}^{\,(-1)^{i+j+k+l}}.
$$
Here, $ 2 \leq i \leq d$ and $2 \leq j \leq n-d$.
These substitutions transform any
monomial in the unknowns $y_{i,j}$ that is invariant into a monomial in the
basic invariants (\ref{eq:birationalpara}).
In particular, for every square submatrix $Z$ of the matrix $\,Y$,
this writes ${\rm det}(Z)^2$ as a polynomial
in $\alpha,\beta_j,\gamma_i,\kappa_{ij}$.
\end{proposition}

\begin{example}[$n=6,d=3$]
The  parametrization (\ref{eq:birationalpara}) is given by
$ \alpha = y_{11}^2, \beta_2 = y_{12}^2, \beta_3 = y_{13}^2,
\gamma_2 = y_{2,1}^2, \gamma_3 = y_{3,1}^2,
\kappa_{11} = y_{1,1} y_{1,2} y_{2,1} y_{2,2},\, \smallskip 
\kappa_{12} = y_{1,2} y_{1,3} y_{2,2} y_{2,3},\,
\kappa_{21} = y_{2,1} y_{2,2} y_{3,1} y_{3,2},\,
\kappa_{22} = y_{2,2} y_{2,3} y_{3,2} y_{3,3}. 
$
Solving these nine equations for the $y_{i,j}$ yields
 the monomials with half-integer exponents
in Proposition \ref{prop:inversion}. The first five equations are easy to solve.
For the others,

\begin{footnotesize}
$$ 
y_{2,2} = \frac{ \kappa_{11} }{ \sqrt{ \alpha \beta_2 \gamma_2}},\,\,
y_{2,3} = \frac{ \sqrt{ \alpha \gamma_2} \,\kappa_{12}}{ \sqrt{\beta_3}\, \kappa_{11}},\,\,
y_{3,2} = \frac{ \sqrt{\alpha \beta_2} \,\kappa_{21} } {\sqrt{\gamma_3}\, \kappa_{11}} ,\,\,
y_{3,3} = \frac{ \sqrt{\beta_3 \gamma_3} \,\kappa_{11} \kappa_{22} } {\sqrt{\alpha} \,\kappa_{12} \kappa_{21}}.
$$ \end{footnotesize}
By substituting these expressions into the $3 {\times} 3 $
matrix $Y = (y_{i,j})$, we find that
$ {\rm det}(Y)^2 $~equals
\begin{footnotesize} $$ \!\! \frac{
(\alpha^2 \beta_2 \gamma_2 \kappa_{12}^2 \kappa_{21}^2-\alpha \beta_2 \beta_3 \gamma_2 \kappa_{11} \kappa_{12} \kappa_{21}^2
-\alpha \beta_2 \gamma_2 \gamma_3 \kappa_{11} \kappa_{12}^2 \kappa_{21}
+\beta_2 \beta_3 \gamma_2 \gamma_3 \kappa_{11}^3 \kappa_{22}-\beta_3 \gamma_3 \kappa_{11}^4 \kappa_{22}
+\beta_3 \gamma_3 \kappa_{11}^3 \kappa_{12} \kappa_{21})^2}
{ \alpha \,\beta_2  \, \beta_3\, \gamma_2 \,\gamma_3 \, \kappa_{11}^4 \,\kappa_{12}^2 \, \kappa_{21}^2} 
$$ \end{footnotesize}
Likewise, the squares of all $2 \times 2$ minors  of $\,Y$ are Laurent polynomials
in the new parameters $\alpha,\beta_2,\beta_3,\gamma_2,\gamma_3,\kappa_{11},\kappa_{12}, \kappa_{21}, \kappa_{22}$.
These squared minors are probabilities in our DPP model.
\end{example}

\section{The projection Grassmannian}

We now examine the affine variety
in $\RR^{\binom{n+1}{2}}$ that is
defined by the entries  of $P^2 - P$.
Let $I_n$ denote the ideal generated by
these $\binom{n+1}{2}$ quadratic polynomials
in the $\binom{n+1}{2}$ entries $p_{ij}$
of the symmetric matrix $P$.
By fixing the trace of $P$, we
obtain the projection Grassmannian.

\begin{definition}[Projection Grassmannian]
  The {\em projection Grassmannian} is defined as
\begin{equation}
\label{eq:affinegrass} {\rm pGr}(d,n) \,\, = \,\, V\bigl( \,I_n + \langle {\rm trace}(P) - d \rangle \,\bigr). 
\end{equation}
\end{definition}

This is an irreducible affine variety of dimension $d(n-d)$.
We now prove that the given equations generate prime ideals in the polynomial ring
$\CC[P]$ whose unknowns are the $\binom{n+1}{2}$ entries of $P = (p_{ij})$.
In particular,  eqn.~(\ref{eq:affinegrass}) gives
 the ideal of the projection Grassmannian.

\begin{theorem} \label{thm:radical}
The ideal $I_n$ is radical, and
$I_{n,d} \coloneqq I_n + \langle \,{\rm trace}(P) - d \,\rangle $ is prime for all $d$.
\end{theorem}

\begin{proof} The ideal $I_n$ is generated by the following quadratic polynomials in
the ring $\CC[P]$:
\begin{equation}
   f_{ij} \,\,\coloneqq \,\,p_{ij} 
   - p_{i1} p_{1j} - p_{i2} p_{2j} - \cdots - p_{1n} p_{nj}
   \quad \text{ for  }\, 1 \leq i \leq j \leq n.
\end{equation}
 Let $X_n $ denote the subscheme of 
    the affine space $ {\rm Spec}(\CC[P]) = \mathbb{A}_\CC^{\binom{n+1}{2}} $ defined by $I_n$.
    Moreover, let $X_{n,d}$ denote the subscheme  of $X_n$  defined by the ideal
    $   I_{n,d} $. Both $X_n$ and $X_{n,d}$ are affine schemes,
    so ``radical'' is equivalent to ``reduced.''
It is our plan to prove that $X_n$ and $X_{n,d}$ are smooth. 
Since smooth schemes are reduced, the first statement follows. The fact that $X_{n, d}$ is irreducible will then imply that
$I_{n,d}$ is a prime ideal in the polynomial ring  $\CC[P]$.

    To prove that $X_n$ is smooth, we will show that $X_n$ is smooth at every closed point in $X_n(\CC)$. We shall use the action of the orthogonal group
    $O(n) = O(n,\CC)$ on symmetric $n \times n$ matrices by conjugation.
    This action preserves the ideal $I_n$ and it therefore induces an action 
    on $X_n(\CC)$.   We claim that  this action has  $n+1$ orbits, namely
      the rank loci  $X_{n,d}(\CC)$ for $d=0,1,\ldots,n$.
    Note that $X_{n,d}(\CC) $ consists of complex symmetric $n \times n$ matrices of rank $d$ that are idempotent.    The $O(n)$ action on $X_n(\CC)$ preserves matrix rank, so it fixes $X_{n,d}(\CC)$.

    We show that the $O(n)$ action is transitive on $X_{n,d}(\CC)$.
    Every matrix in $X_{n,d}(\CC)$ is diagonalizable: its minimal polynomial must divide $z^2 - z$ and is therefore square-free.
      Thus all matrices in $X_{n,d}(\CC)$ are diagonalizable to the
      same diagonal matrix $M_d$. Here, $M_d$ is the matrix with entry one in
      positions $(1,1),(2,2),\ldots,(d,d)$ and all other entries zero. Now
we can use a result from linear algebra, found in
      \cite[Theorem XI.4]{Gantmacher}, which states that
      every symmetric complex diagonalizable matrix can be diagonalized by a orthogonal matrix, so
      every matrix in $X_{n,d}(\CC)$ is conjugate to $M_d$ under the action of the
      orthogonal group $O(n)$.
      In other words, $X_{n,d}(\CC)$ is the orbit of $M_d$ under the conjugation action of $O(n)$.
 
      We will now show that  the affine scheme $X_{n,d}$  is smooth at the closed point $M_d$.
      We know that $X_{n,d}$ is irreducible of dimension $d(n-d)$. Hence, using the Jacobian criterion (see, for example, ~\cite[Lemma 6.27]{GoertzWedhorn}), it suffices to show
      that the total derivatives $d f_{ij}$, evaluated at the point $M_d$, span a 
    space of dimension at least $\binom{n+1}{2} - d(n-d)$.

We compute        $\,d(f_{ij}) \,=\, d(p_{ij} - \sum^n_{k=1} p_{ik} p_{kj}) \,=\, dp_{ij} - \sum^n_{k=1} (p_{ik} d p_{kj} + p_{kj} d p_{ik})$. At the point $M_d$,   we have $p_{ij} = 0$ unless $i = j \leq d$,
in which case we have    $p_{ij} = 1$. This implies
$$ 
d(f_{ij})|_{M_d} \,\,=\,\, \begin{cases}
 - dp_{ij} & {\rm for} \,\,    i \leq j \leq d, \\
     \quad \, 0   & {\rm for} \,\,    i \leq d < j ,\\
  \phantom{-}    dp_{ij} & {\rm for} \,\,    d < i \leq j . 
\end{cases}
$$
We conclude that    the $d(f_{ij})|_{M_d}$ span a space of dimension $\binom{d+1}{2} + \binom{n-d+1}{2} =
 \binom{n+1}{2} - d(n-d)$. Put differently, the evaluation of the
 $\binom{n+1}{2} \times \binom{n+1}{2}$ Jacobian matrix of $(f_{11},f_{12}, \ldots,f_{nn})$ 
at the point $M_d$ has entries in $ \{-1,0,1\}$, and the rank of that matrix
   is at least  $\binom{n+1}{2} - d(n-d)$.
   This shows that $X_{n,d}$ is smooth at $M_d$, and hence $X_{n,d}$ is smooth
   at every closed point.

Now, to get the same statement for $X_n$, we note that
only the generators of $I_n$ were used in the calculation above.
This verifies that $X_n$ is smooth at $M_d$. Since
      $X_n(\CC)$ is the disjoint union of the orbits $X_{n,d}(\CC)$, for $d=0,1,\ldots,n$,
      we conclude that $X_n$ is a smooth scheme.
      
      We finally use  that smooth schemes are reduced;
      see, e.g.,~\cite[Theorem 6.28, Proposition B.77 (2)]{GoertzWedhorn}. Furthermore, an affine scheme is reduced if and only if its
ideal is radical \cite[Proposition 3.27]{GoertzWedhorn}. This shows that $I_n$ is a radical ideal in $\CC[P]$.
By the same reasoning, $I_{n,d}$ is a radical ideal. Since
$X_{n,d}$ is irreducible, we conclude that $I_{n,d}$ is a prime ideal in $\CC[P]$.
\end{proof}

\begin{corollary}
The radical ideal $I_n$ has $n+1$ minimal primes, one for each Grassmannian:
\begin{equation}
\label{eq:primedec}
 I_n \,\, = \,\, \bigcap_{d=0}^n \,I_{n,d}. 
\end{equation}
\end{corollary}

\begin{proof}
From the previous proof we have the following
 disjoint decomposition of varieties:
$$ \qquad X_n \,=\,V( I_n) 
\,\,\,=\,\,\, \,\bigcup_{d=0}^n X_{n,d}
\,\, = \,\, \bigcup_{d=0}^n \,V \bigl(\, I_{n,d}\,\bigr). $$
We now apply the ideal operator on both sides. This turns the union into an intersection of ideals.
Since all ideals are radical, by Theorem \ref{thm:radical}, Hilbert's Nullstellensatz implies
  (\ref{eq:primedec}).
\end{proof}

The Grassmannian ${\rm Gr}(d,n)$ is  isomorphic to the Grassmannian ${\rm Gr}(n-d,n)$.
We can see this isomorphism  very nicely for the
projection Grassmannians ${\rm pGr}(d,n)$ and ${\rm pGr}(n-d,n)$:

\begin{remark} \label{rmk:duality}
The linear involution $P \mapsto {\rm Id}_n - P$ fixes the radical ideal $I_n$,
and it switches the prime ideals $I_{n,d}$ and $I_{n,(n-d)}$ that define
 ${\rm pGr}(d,n)$ and ${\rm pGr}(n-d,n)$.
\end{remark}

There is one more linear coordinate change in matrix space which we wish to point out.

\begin{remark}
The projection Grassmannian
${\rm pGr}(d,n)$ is linearly isomorphic to the variety
of orthogonal symmetric matrices of trace $2d-n$.
Indeed, if we replace the matrix
$P$ with $Q = 2P-{\rm Id}_n$, then
the equation $P^2 = P$ becomes $Q^2 = {\rm Id}_n$,
so $Q$ is a symmetric matrix that is orthogonal.
Lai, Lim and Ye \cite{LLY} proposed this embedding of ${\rm pGr}(d,n)$ into the 
orthogonal group $O(n)$
in order to improve numerical stability in
 Grassmannian optimization~\cite{LWY}.
 \end{remark}

After the dimension, the second most important invariant of an embedded variety
is its degree. This is the number of complex intersection points with a generic
affine-linear subspace of complementary dimension.
 We computed the degree of our ideals up to $n=10$.

\begin{proposition} \label{prop:degs}
The degrees of the projection Grassmannians 
 ${\rm pGr}(d,n)$ are as follows:
 \setcounter{MaxMatrixCols}{13}
$$ \begin{matrix}
 & d\,\,= & 0 & 1 & 2 & 3 & 4 & 5 & 6 & 7 & 8 & 9 & 10 \smallskip \\
  n = 3 && 1 & 4 & 4 & 1 \\
  n = 4 && 1 & 8 & 12 & 8 & 1 \\
  n = 5 && 1 & 16 & 40 & 40 & 16 & 1  \\
  n = 6 && 1 & 32 & 140 & 184 & 140 & 32 & 1\\
  n = 7 && 1 & 64 & 504 & 992 & 992 & 504 &  64 & 1 \\
  n = 8 && 1 & 128 &  1848  & 5824 & 7056 & 5824 & 1848 & 128 & 1 \\
  n = 9 && 1 & 256 & 6864 & 36096 & 60864 & 60864 & 36096 & 6864 & 256 & 1\\
  n = 10&& 1 & 512 & 25740 & 232320 &587664 & 672288 &587664 & 232320  & 25740 & 512 & 1
  \end{matrix}
$$
\end{proposition}

\begin{proof}
Rows 3-8  were created with {\tt Macaulay2} \cite{M2}. We used
    {\tt HomotopyContinuation.jl} \cite{BT} for rows 9-10.
Namely, we computed the intersection 
 with affine-linear subspaces of complementary dimension
using the monodromy method in {\tt HomotopyContinuation.jl}.
\end{proof}

\begin{corollary}
The following observations from the table above are valid for all $n \geq 3$:
  \begin{itemize}
\item[(a)] The degrees agree for $(d,n)$ and $(n-d,d)$. \vspace{-0.18cm}
\item[(b)] The degree for $d=0$ equals $1$.                  \vspace{-0.16cm}
\item[(c)] The degree for $d=1$ equals $2^{n-1}$.
\end{itemize}
\end{corollary}
  
 \begin{proof} Statement (a) follows from Remark \ref{rmk:duality}.
 Statement (b) holds because ${\rm pGr}(0,n)$ is just the
 zero matrix. Similarly, ${\rm pGr}(n,n) = \{{\rm Id}_n\}$.
For statement (c), we note that ${\rm pGr}(1,n)$ is the
$(n-1)$-dimensional variety consisting of all symmetric rank $1$
matrices of trace $1$. This is an affine-linear section of the
quadratic Veronese variety $\nu_2(\RR^n)$, so its degree is $2^{n-1}$.
 \end{proof}

\begin{conjecture}\label{conj:29}
The degree of  the $(2n-4)$-dimensional variety ${\rm pGr}(2,n)$ equals
 $\,2\binom{2n-4}{n-2}$.
\end{conjecture}

For each of our ideals, we computed the initial ideal with respect
to the standard weights, where each variable $p_{ij}$ has degree $1$.
This initial ideal contains the trace of $P$ and the entries of $P^2$.
We observed that these  generate the ideal when
$d$ is in the middle between $0$ and~$n$. 

\begin{conjecture} \label{conj:gbden}
Fix $d = \lfloor n/2 \rfloor$ and
consider the Gr\"obner degeneration 
with respect to the partial monomial order given by the usual total degree.
Then the limit of the projection Grassmannian ${\rm pGr}(d,n)$ is the
variety of symmetric $n \times n$ matrices whose square is zero.
The initial ideal equals $\langle P^2 , {\rm trace}(P) \rangle$.
This  ideal is radical, but it is generally not prime.
\end{conjecture}
  
\begin{example}[$n=4$]
The variety of $4 \times 4$ matrices of square zero 
has dimension $4$ and degree $12$ in $\CC^{10}$.
Its radical ideal $\langle P^2 , {\rm trace}(P) \rangle$ is
the intersection of two prime ideals of dimension $4$ and degree $6$,
each generated by $15$ quadrics together with the trace of $P$.
\end{example}
  
  Conjecture \ref{conj:gbden} leads to the following linear algebra question:
  {\em Which complex symmetric matrices have square zero?}
We can show that the variety $V(\langle P^2 , {\rm trace}(P) \rangle)$
is irreducible of dimension $d(d+1)$ for $n=2d+1$ odd,
and it has two components of dimension $d^2$ for $n=2d$ even.
These dimensions are consistent with Conjecture \ref{conj:gbden}.
The common degree of ${\rm pGr}(n,\lfloor n/2 \rfloor)$ and its initial variety
is $4,12,40,184,992,7056,\ldots$.
What is this sequence?

\begin{remark}
After the first posting of our article,
Lim and Ye \cite{LY} proved Conjecture~\ref{conj:29}.
In fact, they derived a   general formula for the degree of the real projection Grassmannian
\cite[Theorem 4.3]{LY}.
Their approach is to view the  projection Grassmannian ${\rm pGr}(d, n)$ as a quotient of the special orthogonal group $SO_n(\CC)$ and compute the degree by decomposing the coordinate ring of the Grassmannian into irreducible $SO_n(\CC)$-modules. 
Conjecture~\ref{conj:29} follows as a corollary.
In \cite[Corollary 4.5]{LY} they give closed formulas for the degree when $d = 3, 4$.
In \cite[Section 5]{LY},
Lim and Ye also give a set-theoretic answer to our Conjecture~\ref{conj:gbden}. 
\end{remark}

\section{The moment map}

In this section, we study the \emph{moment map} $m:{\rm Gr}(d,n)\rightarrow\mathbb{R}^n$. 
Generally, a moment map can be associated to a group acting on a symplectic manifold, and this construction is of interest in symplectic geometry and geometric invariant theory. In context of the Grassmannian, with symplectic structure induced by the Fubini-Study metric, 
the moment map arises naturally from the action of the torus $(\RR^\star)^n$ on the Grassmannian that is induced by stretching the axes of $\mathbb{R}^n$; see \cite[Example 3.5]{kirwan1984}. 
Writing $e_I = \sum_{i\in I} e_i \in \RR^n$ for $I \in \binom{[n]}{d}$, the image of
a point $X = (x_I)$ in
the Grassmannian ${\rm Gr}(d,n) \subset \PP^{\binom{n}{d}-1}$ under the moment map~equals
\begin{equation}
\label{eq:moment2}
m(X) \,\,=\,\, \frac{1}{\sum_J x_J^2} \cdot {\sum_I x_{I}^2 \, e_{I}}.
\end{equation}
Hence, the image of the moment map lies in the {\em hypersimplex} $\Delta(d,n)={\rm conv} \bigl\{ e_I:I\in\binom{[n]}{d} \bigr\}$.
This convex polytope consists of all nonnegative vectors which sum to $d$. The  convexity property,
derived in \cite{GGMS},
is that the moment map is in fact a surjection onto the hypersimplex:
\begin{equation}
\label{eq:convexityproperty}
m({\rm Gr}(d,n)) \,\,=\,\, \Delta(d,n).
\end{equation}

\begin{example}[$n=4,d=2$] \label{ex:mm24}
The moment map $m$ maps the $4$-dimensional Grassmannian ${\rm Gr}(2,4)$ onto the 
octahedron $\Delta(2,4) = {\rm conv}\{e_i+e_j: 1\leq i< j\leq 4 \}$.
The fiber over any interior point of ${\rm Gr}(2,4)$ is a curve. The fiber over
a vertex $e_i+e_j$ has several irreducible components over $\CC$.
But there is only one real point in that fiber, 
namely $e_i \wedge e_j$ in $\PP^5$.
\end{example}

The property (\ref{eq:convexityproperty}) 
remains valid if we restrict $m$
to the positive Grassmannian ${\rm Gr}(d,n)_{>0}$. 
A more refined result is that the torus orbit of any point $X\in{\rm Gr}(d,n)$ is mapped to
its {\em matroid polytope}  ${\rm conv}\{ e_I \,: \, x_I \neq 0 \}$; see \cite[Definition 13.6]{MS}.
 The restriction to positive points in the orbit are mapped bijectively to the open
  matroid  polytope, by~\cite[Theorem~8.24]{MS}.

\begin{example}
Let $X\subseteq\mathbb{R}^n$ be the cut space of a graph $G$, as  in Example \ref{ex:cut space}. The  coordinate $m_i(X)$ of the moment map is 
 the \emph{effective resistance} of edge $i$; see \cite{biggs1997}. Scaling the axes of $\mathbb{R}^n$ corresponds to choosing 
 edge lengths for $G$. The moment map induces a bijection between effective resistances (points in the matroid polytope) 
 and positive choices of edge lengths.
\end{example}

The moment map is another bridge between the
lives of the Grassmannian, and it also appears naturally in the
statistical DPP model that is given by the squared Grassmannian. 
We begin with the projection Grassmannian ${\rm pGr}(d,n)$,
which is an affine variety in $\RR^{\binom{n+1}{2}}$. 

\begin{proposition}
The moment map on ${\rm pGr}(d,n)$ is the projection map onto the diagonal:
\begin{equation}
\label{eq:moment3}
{\rm pGr}(d,n)\ni (p_{ij})_{1\leq i<j\leq n} \,\,\overset{m}{\longmapsto}\,\, (p_{11},p_{22},\dots,p_{nn})\in\mathbb{R}^n.
\end{equation}
\end{proposition}

\begin{proof}
Let $X$ be the Pl\"{u}cker coordinates of a subspace in ${\rm Gr}(d,n)$ and $P$ the projection matrix onto this subspace.
 Following Theorem \ref{thm:PPlucker}, the $i$th coordinate of the moment map 
 (\ref{eq:moment2}) is equal to
$\,
m_i(X) \,=\, \sum_K x^2_{iK}/ \sum_J x^2_J \,=\, p_{ii}$.
This identity establishes the claim.
\end{proof}

\begin{corollary}
The projection of ${\rm pGr}(d,n)$ onto the diagonal is the hypersimplex $\Delta(d,n)$.
\end{corollary}

The image of the complexification
of (\ref{eq:moment3}) is the affine space $\CC^{n-1} = \{z_1 + z_2 + \cdots + z_n = d \}$.
In the projective life of the Grassmannian, we identify this target space
with $\PP^{n-1}$.
The {\em projective moment map} is the following linear projection from the
 squared Grassmannian:
\begin{equation}
\label{eq:moment4}
{\rm Gr}(d,n) \,\rightarrow \,{\rm sGr}(d,n)\, \dashrightarrow \,\PP^{n-1}\, , \,\,
(x_I)\, \mapsto\, (q_I) \, \mapsto\,
\bigl(\sum_{K} q_{1K}: \sum_{K} q_{2K} : \,\cdots\, : \sum_K q_{nK} \bigr).
\end{equation}

\begin{example}[$d=2$]
By Theorem \ref{thm:conca},  ${\rm sGr}(2,n)$ comprises
symmetric $n\times n$ matrices $Q$ with zero diagonal and rank $\leq 3$. The moment map
(\ref{eq:moment4})
takes $Q$ to its row sum vector $m(Q) = (\sum_j q_{ij})_{1\leq i\leq n}$.
The image is the second hypersimplex $\Delta(2,n)$,  viewed projectively~in~$\PP^3$.
\end{example}

In Section 4, we identified  ${\rm sGr}(d,n)$ with a discrete statistical model,
namely the projection DPP model. In this setting,
the moment map serves as a  marginalization operator.
 
\begin{corollary}
Let $\mu$ be a projection DPP on $\binom{[n]}{d}$.  The $i$th coordinate of the moment 
map is the marginal probability that a random $d$-set in $[n]$ contains the element $i$.
The space of all marginal distributions for the projection DPP model on $\binom{[n]}{d}$ is
 the hypersimplex $\Delta(d,n)$.
\end{corollary}

We now come to the fibers of the moment map. 
Fix $z \in \RR^n$ with $\sum_{i=1}^n z_i = d$. The
fiber over $z$ in the projection Grassmannian, denoted
$ {\rm pGr}(d,n)_z $, is the variety defined by 
\begin{equation}
\label{eq:pGrfiber}
 \quad P^2  = P \quad {\rm and}  \quad {\rm diag}(P) = z \quad \hbox{ for symmetric $n \times n$ matrices $P$.} 
\end{equation}

\begin{lemma}
For a general point $z$, the affine fiber $ {\rm pGr}(d,n)_z $
is an irreducible variety of dimension $(n-d-1)(d-1)$, and its
degree equals that of $ {\rm pGr}(d,n)$, given in Proposition~\ref{prop:degs}.
\end{lemma}

\begin{proof}
The moment map takes the Grassmannian of dimension
$d(n-d)$ onto an affine space of dimension  $n-1$. The dimension of
the generic fiber is the difference between these~two numbers.
The set of symmetric matrices $P$ with ${\rm diag}(P) = z$
is a generic affine-linear space. It intersects
${\rm pGr}(d,n)$ in an irreducible variety of the same
degree, by Bertini's Theorem.
\end{proof}

\begin{example}[$d=2$] \label{ex:affinefiber}
The general fiber ${\rm pGr}(2,n)_z$
is a variety of dimension $n-3$
and degree $2 \binom{2n-4}{n-2}$.
This formula is proved for $n \leq 10$, and it is
Conjecture \ref{conj:29} for $n \geq 11$. In particular,
for $n=4$ we obtain a curve of degree $12$,
and for $n=5$ it is a surface of degree~$40$.
\end{example}

The fibers of the projective moment map
in (\ref{eq:moment4}) are denoted by
$ {\rm Gr}(d,n)_z $ and 
$ {\rm sGr}(d,n)_z $ respectively, where now $z \in \PP^{n-1}$.
Both fibers are projective varieties of dimension $(n-d-1)(d-1)$ in
$\PP^{\binom{n}{d}-1}$, and the latter is the image of the former under
the squaring morphism. 

\begin{remark}
The ideal of $ {\rm sGr}(d,n)_z $
 is obtained from ${\rm sGr}(d,n)$ by adding the $2 \times 2$ minors~of
\begin{equation}
\begin{bmatrix}
\label{eq:2byn}
\sum_K q_{1K} & \sum_K q_{2K} & \cdots &  \sum_K q_{nK} \\
 z_1  & z_2 & \cdots & z_n  \\
\end{bmatrix},
\end{equation}
and then saturating by the ideal 
generated by  the first row in (\ref{eq:2byn}).
The ideal for $ {\rm Gr}(d,n)_z $ is found in the same manner:
we take the Pl\"ucker ideal of ${\rm Gr}(d,n)$,
augmented by the $2 \times 2$ minors of the matrix (\ref{eq:2byn})
with $q_{iK}$ replaced by $x_{iK}^2$, and we  saturate this
by  the first row.
\end{remark}

\begin{example}[$d=2$] \label{ex:fibers}
The general fiber ${\rm sGr}(2,n)_z$ has dimension~$n-3$ and
 lies in a linear space of codimension $n-1$ in $\PP^{\binom{n}{2}-1}$.
For $n=4$, it is a plane cubic. For $n=5$, it is a surface
of degree $10$ in a $\PP^5$ with ideal generated by $10$ cubics.
For $n=6$, it is a threefold of degree $35$
in a $\PP^9$, defined by $50$ cubics.
For $n=7$, the degree is $126$ and there are $175$ cubics.

The inverse image of ${\rm sGr}(2,n)_z$ under the squaring map
is the fiber ${\rm Gr}(2,n)_z$. This variety linearly spans $\PP^{\binom{n}{2}-1}$,
it has the same degree as the affine fiber
${\rm pGr}(2,n)_z$ in Example \ref{ex:affinefiber}.
For $n=4$, it is a curve of degree $12$, whose 
ideal is generated by $4$ quadrics
and $2$ cubics. For $n=5$, it is a surface of degree $40$,
defined by $9$ quadrics, $4$ cubics and $5$ quartics.
For $n=6$, it is a threefold of degree $140$, defined
by $20$ quadrics, $10$ cubics and $15$ quartics.
For $n=7$ we get a fourfold of degree $504$, defined
by $41$ quadrics, $20$ cubics and $105$ quintics.
\end{example}

If the point $z$ is chosen in a special position in
the hypersimplex $\Delta(d,n)$, then
the fiber of the moment map can be reducible,
and a beautiful  combinatorial structure
emerges. We conclude this paper by
showing this, in a  case study for the smallest scenario.
Namely, we fix $d=2$ and $n=4$, and we focus on the
fibers of the projective moment map
${\rm Gr}(2,4) \dashrightarrow \PP^3$.
The combinatorics we now describe is the same for the
affine moment map ${\rm pGr}(2,4) \rightarrow \CC^3$.

\begin{example}($n=4, d=2$)
The chamber complex of the octahedron $\Delta(2,4)$
arises by slicing $\Delta(2,4)$
with the three planes spanned by quadruples of vertices.
It consists of $8$ tetrahedra, $12$ triangles,
$6$ edges, and one  vertex, located at 
$ (1/2,1/2,1/2,1/2)$.
We distinguish four cases, depending on the cell of the 
chamber complex that has $z$ in its relative interior.

\begin{figure}[h]
\centering
\includegraphics[width = 16cm]{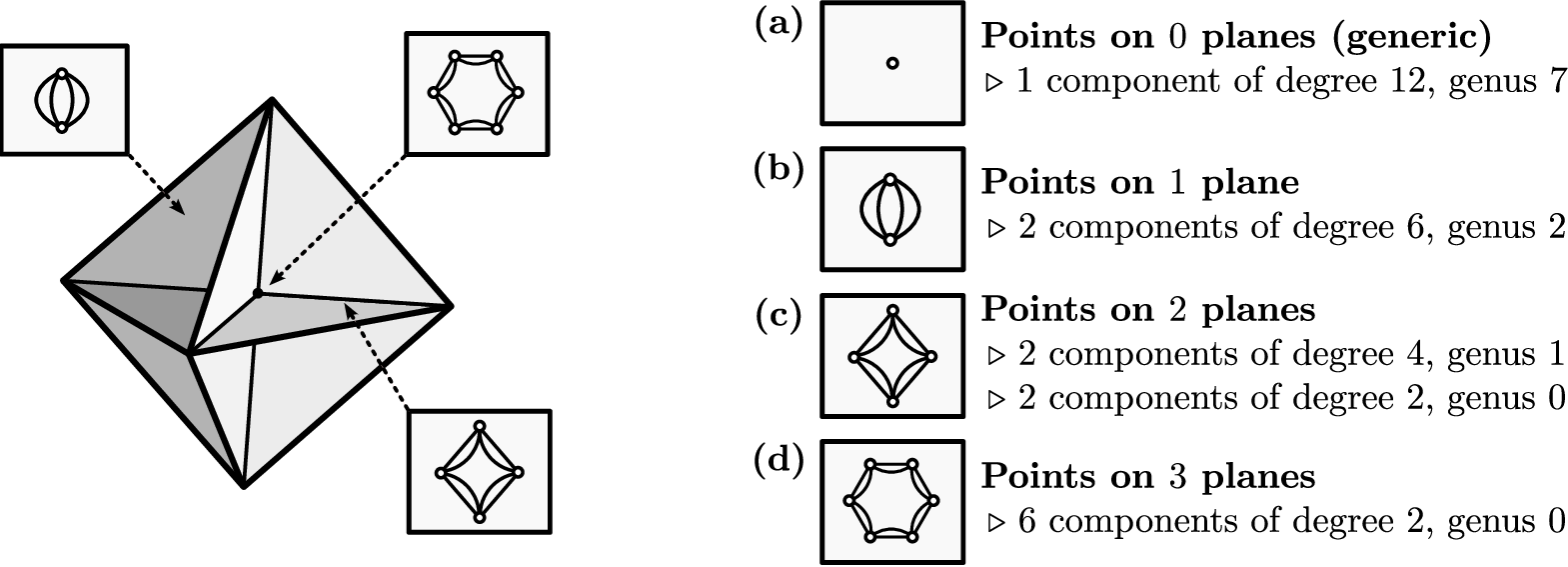} 
\caption{Moment map fibers over the chamber complex of the octahedron
$\Delta(2,4)$.}
\label{fig:octafib}
\end{figure}

\begin{itemize}
\item[(a)] The generic fiber ${\rm Gr}(2,4)_z$ is an
irreducible smooth curve
of degree $12$ and genus $7$. \vspace{-0.08in}
\item[(b)] If $z$ is in a triangle of the chamber complex, then its fiber ${\rm Gr}(2,4)_z$
has two irreducible components, each a curve of degree 
$6$ and genus $2$, which intersect in four points.  \vspace{-0.08in}
\item[(c)] If $z$ is on an edge, then there are
four components: two $\PP^1$'s and two elliptic curves. 
 \vspace{-0.08in}
\item[(d)] The central fiber ${\rm Gr}(2,4)_z$ is a ring of six
$\PP^1$'s: consecutive pairs touch at two points.
\end{itemize}

 Figure~\ref{fig:octafib} shows the dual graphs for the four types of fibers.
 Their vertices are the
irreducible components of the curve. Edges
 represent intersections
between pairs of components.
Figure \ref{fig:endler} shows
the Riemann surfaces corresponding to the 
curves of type (a) and (d).
The generic fiber (a) is
a Riemann surface of genus seven.
The most special fiber (d) is a ring
of six Riemann spheres, touching consecutively
at pairs of points.  Passing to nearby generic fibers
smoothes the $12$ singular points and creates the Riemann surface of genus seven.

The squaring map ${\rm Gr}(2,4)_z \rightarrow {\rm sGr}(2,4)_z$
takes our surface of genus seven onto a torus. This is the plane
cubic curve in Example \ref{ex:fibers}. By the
Riemann-Hurwitz Theorem, the map has $12$ ramification points,
and these correspond to the $12$ singular points seen in type (d).
For the dual graph (d) in Figure \ref{fig:octafib}, 
we obtain the natural $2$-to-$1$ map onto the hexagon.

\begin{figure}[h]
\centering \vspace{-0.1in}
 \includegraphics[width = 15cm]{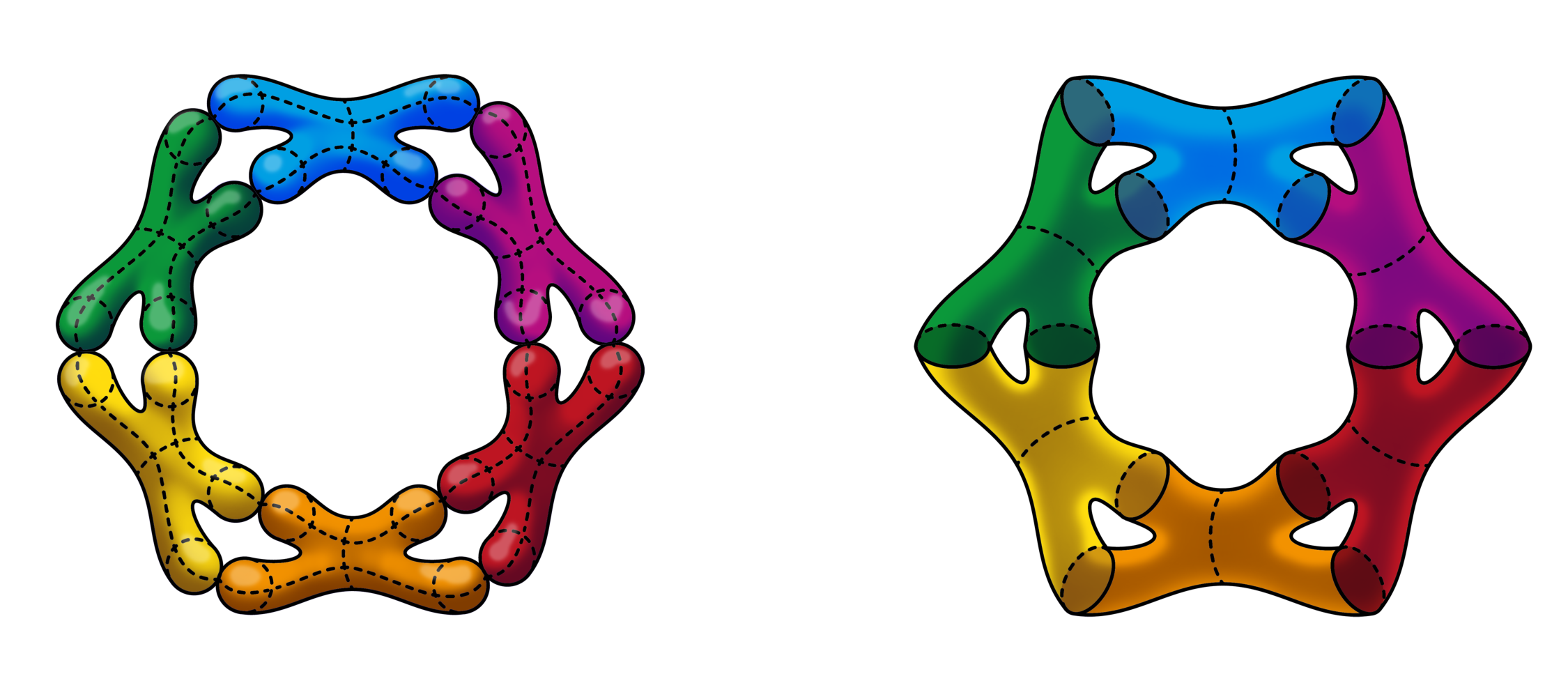} 
\vspace{-0.16in}
\caption{The fiber (d) over the midpoint of $\Delta(2,4)$, shown on the left, is a ring
of six Riemann spheres. The general fiber (a), shown on the right,
is a Riemann surface of genus~seven.}
\label{fig:endler}
\end{figure}
\end{example}

\noindent {\bf Acknowledgment.}
We thank Aldo Conca for helping us with
the proof of Theorem \ref{thm:conca}.

		\bigskip
		\bigskip
		
		\footnotesize
		\noindent {\bf Authors' addresses:}
		
		\smallskip
		
		\noindent Karel Devriendt, MPI-MiS Leipzig
		\hfill \url{karel.devriendt@mis.mpg.de}

		\noindent Hannah Friedman, UC Berkeley
		\hfill \url{hannahfriedman@berkeley.edu}

		\noindent  Bernhard Reinke, MPI-MiS Leipzig
		\hfill \url{bernhard.reinke@mis.mpg.de}

		\noindent  Bernd Sturmfels, MPI-MiS Leipzig
		\hfill \url{bernd@mis.mpg.de}

\end{document}